\documentclass[
twoside,headrule,11pt,reqno]{amsart}
\usepackage[mathscr]{eucal}
\usepackage{amsfonts,amssymb,amsmath,%
amsthm,enumerate,mdwlist,url}
\usepackage[oldenum]{paralist}

\footskip=10mm\parskip 0.2cm\addtocounter{page}{0}
scaled \magstep0

\usepackage[flushmargin]{footmisc}
\theoremstyle{plain}
\newtheorem{theorem}{Theorem}
\newtheorem{proposition}[theorem]{Proposition}
\newtheorem*{lemma}{Lemma}

\theoremstyle{definition}
\newtheorem*{question}{Question}
\newtheorem*{questions}{Questions}

\newtheorem*{examples}{Examples}
\newtheorem*{remark}{Remark}

\usepackage{hyperref}

\def\after{\raise.25ex\hbox{$\scriptstyle\thinspace\circ\thinspace$}}
\def\Bd{\partial}
\def\C{\mathbb{C}}
\def\d{\delta}
\def\e{\varepsilon}
\def\from{\colon\thinspace}
\def\Hopf{\mathscr{H}}
\def\R{\mathbb{R}}
\def\suchthat{\thinspace\colon\thinspace}
\def\Suchthat{\mid}

\def\zero{\mathbf{0}}

\newcommand{\bydef}[1]{\emph{#1}}

\begin{document}
\bibliographystyle{amsplain}


\def\evenhead{{\protect\centerline{\textsl{\large{Lee Rudolph}}}\hfill}}

\def\oddhead{{\protect\centerline{\textsl{\large{Whitehead's Integral Formula
and Isolated Critical Points}}}\hfill}}

\pagestyle{myheadings} \markboth{\evenhead}{\oddhead}

\thispagestyle{empty} \noindent{{\small\rm Pure and Applied
Mathematics Quarterly\\ Volume 6, Number 2\\ (\textit{Special Issue:
In honor  of \\ Michael Atiyah and Isadore Singer)\\
\textup{545--554, 2010 (ArXival redaction)}} \vspace*{1.5cm} \normalsize

\begin{center}
\Large {\bf Whitehead's Integral Formula, Isolated Critical
\\Points, and the Enhancement of the Milnor Number}
\end{center}

\begin{center}
\large Lee Rudolph
\end{center}

\begin{center}
\it for Is Singer on his $85^\textit{th}$ birthday
\end{center}
\footnotetext{Received March 7, 2008.\\During various parts of the
period when this research was performed, the author was partially
supported by the Fonds National Suisse, and by NSF awards
DMS-8801915, DMS-9504832, DMS-0308894, and IIS-0713335.}

\medskip
\begin{center}
\begin{minipage}{5in}
\noindent {\bf Abstract:} J.\ H.\ C.\ Whitehead gave an elegant
integral formula for the Hopf invariant $\Hopf(\mathbf{p})$ of a
smooth map $\mathbf{p}$ from the $3$-sphere to the $2$-sphere.
Given an open book structure $\mathbf{b}$ on the $3$-sphere (or,
essentially equivalently, an isolated critical point of a map $F$
from $\R^4$ to $\R^2$), Whitehead's formula can be ``integrated
along the fibers'' to express $\Hopf(\mathbf{p})$ as the integral of
a certain $1$-form over $S^1$.  In case $\mathbf{p}$ is
geometrically related to $\mathbf{b}$ (or $F$)---for instance, if
$\mathbf{p}$ is the map (one component of the fiberwise generalized
Gauss map of $F$) whose Hopf invariant $\lambda(K)$ is the
``enhancement of the Milnor number'' of the fibered link $K\subset
S^3$ associated to $F$ (or $\mathbf{b}$), previously studied by the
author and others---it might be hoped that this $1$-form has
geometric significance.  This note makes that hope somewhat more
concrete, in the form of several
speculations and questions.
\end{minipage}
\end{center}

\section{Whitehead's integral formula and isolated critical points}
\label{WIF and ICP}

Let $\mathbf{p}\from S^3\to S^2$ be smooth. According to Whitehead
\cite{Whitehead} (see also \cite{Whitney,BottTu}), the Hopf
invariant $\Hopf(\mathbf{p})$ of $\mathbf{p}$ can be calculated as
follows. Let $\Omega$ be the volume $2$-form on $S^2$, normalized to
give $S^2$ volume $1$. Since $\Omega$ is closed, so is its pull-back
$\Omega_\mathbf{p} = \mathbf{p}^*(\Omega)$; since
$H^2(S^3;\R)=\{0\}$, $\Omega_\mathbf{p}$ is in fact exact. Let
$\eta_\mathbf{p}$ be a $1$-form on $S^3$ with
$d\eta_\mathbf{p}=\Omega_\mathbf{p}$.  If $g$ is any $0$-form on
$S^3$, then also $d(\eta_\mathbf{p}+dg)=\Omega_\mathbf{p}$, and
since $H^1(S^3;\R)=\{0\}$, every solution $\eta$ of
$d\eta=\Omega_\mathbf{p}$ is of the form $\eta_\mathbf{p}+dg$.
Thus, pointwise, the $3$-form
$\eta_\mathbf{p}\wedge\Omega_\mathbf{p}$ depends strongly on the
choice of $\eta_\mathbf{p}$. Nonetheless, Whitehead showed that its
integral is independent of this choice:
\begin{equation}\label{WIF}
\tag{WIF}\int\limits_{\textstyle S^3} \eta_\mathbf{p}%
  \wedge\Omega_\mathbf{p} = \Hopf(\mathbf{p}).
\end{equation}

It is clear that \eqref{WIF} remains valid, \textit{mutatis
mutandis}, if subjected to any of various modifications: the target
of $\mathbf{p}$ could be $\R^3\setminus\{\zero\}$ or
$(\R^3\setminus\{\zero\})\times\R$, the domain of $\mathbf{p}$ could
be a punctured neighborhood of $\zero$ in $\R^4$, and (in the latter
case) the integration in \eqref{WIF} could be over any $3$-cycle in
the domain of $\mathbf{p}$ dual to the puncture at $\zero$.

Let $U$ be a neighborhood of $\zero$ in $\R^4$, $F=(f,g)\from
U\to\R^2$ a map with $F(\zero)=\zero$ such that $F$ is smooth on
$U\setminus\{\zero\}$ and $F$ is continuous at $\zero$.  We will say
that $\zero$ \bydef{is an isolated critical point of $F$} provided
that there exists $\d>0$ such that, if $0<\|\mathbf{x}\|\le\d$, then
the derivative $DF(\mathbf{x})$ has rank $2$; and (following
Kauffman \& Neumann \cite{Kauffman-Neumann}) we will say that
$\zero$ is a \bydef{tame} isolated critical point of $F$ if,
further, for all sufficently small $\d>0$,
\begin{inparaenum}[(a)]
\item\label{KN1}
$F^{-1}(\zero)$ and $\d S^3$ intersect transversely, and
\item\label{KN2}
there exists $\e(\d)>0$ such that, if $0<\e\le\e(\d)$, then
$F^{-1}(\e D^2) \cap \d D^4$ is a topological $4$-disk that is
smooth except for corners along $F^{-1}(\e S^1)\cap \d S^3$.
\end{inparaenum}

\begin{examples}\label{icp examples}
\begin{inparaenum}[(a)]
\item\label{regular point is icp}
If $F$ is smooth at $\zero$ and $\zero$ is a regular point of $F$,
then $\zero$ is a tame isolated critical point of $F$.
\item\label{Milnor icps}
If $F$ is (the real-polynomial mapping underlying) a
complex-polynomial mapping $(\C^2,\zero)\to(\C,0)$ without repeated
factors, then $\zero$ is a tame isolated critical point of $F$. The
differential topology of this case (and its higher-dimensional
generalization), first studied by Milnor
\cite{Milnor:singular-points}, continues to be a lively topic of
investigation.
\item\label{real case}
Let $F\from(\R^4,\zero)\to(\R^2,\zero)$ be a real-polynomial
mapping.  For a generic such $F$, the real-algebraic set of critical
points of $F$ is $1$-dimensional at each of its points, so $\zero$
is an isolated critical point of $F$ if and only if it is a regular
point. Nonetheless, there exist many real-polynomial mappings $F$
for which $\zero$ is a non-regular isolated critical point
(necessarily tame, \cite{Kauffman-Neumann}). Among these mappings,
those that underlie complex-polynomial mappings are neither the
only, nor necessarily the most interesting, examples.  (A wide
variety of essentially non-complex examples are known; see
\cite{Looijenga,Perron,Rudolph:isocp-1,%
PichonSeade,Pichon,BodinPichon,HirasawaRudolph}.)
\end{inparaenum}
\end{examples}

\begin{proposition}\label{ICP integral formula}
If $F=(f,g)\from U\to\R^2$ has a tame isolated critical point at
$\zero$, and $\mathbf{p}\from U\setminus\{\zero\}\to
\R^3\setminus\{\zero\}$ is a smooth map, then for all sufficiently
small $\d>0$,
\begin{align}
\Hopf(\mathbf{p})=&
\smash[t]{%
                   \lim_{\e\to 0}%
                   \int\limits_{\textstyle{\e S^1}}%
                   \hskip.1in
                   \int\limits_{\textstyle{%
                        {F^{-1}(\e\mathbf{u}})\cap\d D^4}}%
                   \hskip-.4in\eta_{\mathbf{p}}\wedge\Omega_{\mathbf{p}}
}
\label{ICPIF--first form}
\end{align}
\end{proposition}

\begin{proof}
By hypothesis \eqref{KN2} in the definition of tameness, the
$3$-manifold-with-corners $\Bd(F^{-1}(\e D^2) \cap \d D^4)$ (with
its natural orientation) is a $3$-cycle homologous to $\d S^3$ in
$U\setminus\{\zero\}$, provided that $\d$ and $\e$ are sufficiently
small.  As a $3$-cycle, $\Bd(F^{-1}(\e D^2) \cap \d D^4)$ is the sum
of two $3$-chains, $C_1(\d,\e)=F^{-1}(\e S^1) \cap \d D^4$ and
$C_2(\d,\e)= F^{-1}(\e D^2) \cap \d S^3$; the latter, by hypothesis
\eqref{KN1} in the definition of tameness, is a neighborhood of a
smooth link $L=\d S^3\cap F^{-1}(\zero)$ in $\d S^3$, and it is
clear that for a fixed $\d$ this neighborhood shrinks down to $L$ as
$\e$ goes to $0$.  By the comments after \eqref{WIF},
\begin{equation*}
\Hopf(\mathbf{p})=\hskip-12pt
                  \int\limits_{\textstyle C_1(\d,\e)}
                  \hskip-12pt \eta_\mathbf{p}%
                  \wedge\Omega_\mathbf{p}+
                  \hskip-12pt \int\limits_{\textstyle C_2(\d,\e)}
                  \hskip-12pt  \eta_\mathbf{p}%
                  \wedge\Omega_\mathbf{p},
\end{equation*}
so
\begin{equation*}
\Hopf(\mathbf{p})=\lim\limits_{\e\to 0}\
                  \hskip-12pt \int\limits_{\textstyle C_1(\d,\e)}
                  \hskip-12pt \eta_\mathbf{p}%
                  \wedge\Omega_\mathbf{p};
\end{equation*}
this becomes \eqref{ICPIF--first form} after integration along the
fiber.
\end{proof}

\section{Whitehead's integral formula and open books}
\label{WIF and OB} Tame isolated critical points are closely related
to open-book structures.  Here, an \bydef{open book} on $S^3$ is a
smooth map $\mathbf{b}\from S^3\to\R^2$ such that $\zero$ is a
regular value of $\mathbf{b}$
and $\mathbf{b}/\|\mathbf{b}\|\from S^3\setminus %
\mathbf{b}^{-1}(\zero)\to S^1$ is a fibration.  Given an open book
$\mathbf{b}$, it is easy to check that $\zero$ is a tame isolated
critical point of $\mathop{co}(\mathbf{b})\from\R^4\to\R^2$, where
$\mathop{co}(\mathbf{b})(t\mathbf{x})=%
t\mathbf{b}(\mathbf{x})$ for all $\mathbf{x}\in S^3$ and $t\ge0$.
Conversely \cite{Kauffman-Neumann}, if $\zero$ is a tame isolated
critical point of $F \from (U,\zero)\to(\R^2,\zero)$, then there is
an open book $\mathbf{b}_F$ (unique up to an obvious equivalence
relation) such that the isolated critical points at $\zero$ of $F$
and of $\mathop{co}(\mathbf{b}_F)$ are equivalent (again, in an
obvious sense).

The \bydef{binding} of the open book $\mathbf{b}$ is
$L_\mathbf{b}=\mathbf{b}^{-1}(\zero)$. The
\bydef{$\theta^\textrm{th}$ page} of $\mathbf{b}$
is $S_\mathbf{b}(\theta) =\mathbf{b}^{-1}(\{t(\cos(\theta),%
\sin(\theta)) \Suchthat t\ge 0\})$.  It is immediate from the
definitions that $L_\mathbf{b}$ is a smooth link (naturally oriented
by the orientations of $S^3$ and $S^2$) and that each
$S_\mathbf{b}(\theta)$ is a Seifert surface bounded by
$L_\mathbf{b}$.  In fact, the binding of an open book is a fibered
link, and each page is a fiber surface for that link.

\begin{proposition}\label{open book integral formula}
If $\mathbf{b}\from S^3\to\R^2$ is an open book, and
$\mathbf{p}\from S^3\to S^2$ is a smooth map, then
\begin{align}\Hopf(\mathbf{p})
                 &=\lim_{\e\to 0}\hskip6pt%
                   \int\limits_{\textstyle{S^1}}%
                   \hskip.1in\int\limits_{\textstyle{S_\mathbf{b}(\theta)
                                 \setminus
                                 \mathbf{b}^{-1}(\e(\mathrm{Int}\,D^2))}}
                   \hskip-.65in\eta_{\mathbf{p}}\wedge\Omega_{\mathbf{p}}
\label{OBIF--first form}
\end{align}
\end{proposition}

\begin{proof} This is entirely analogous to
Proposition~\ref{ICP integral formula}.
\end{proof}

\section{The enhancement of the Milnor number}
\label{lambda}

As in \textsection\ref{WIF and ICP}, let
$F=(f,g)\from(U,\zero)\to(\R^2,\zero)$ be smooth on
$U\setminus\{0\}\subset\R^4$ and continuous at $\zero$. Easily,
$\zero$ is an isolated critical point of $F$ if and only if there
exists $\d>0$ such that, if $0<\|\mathbf{x}\|\le\d$, then $(df\wedge
dg)(\mathbf{x})\ne 0$. In this case, the self-dual (resp.,
anti-self-dual) $2$-form $df\wedge dg+\star(df\wedge dg)$ (resp.,
$df\wedge dg-\star(df\wedge dg)$), where $\star$ is the Hodge star
operator with respect to the flat metric on $\R^4$, is
non-degenerate; writing
\begin{multline}
df\wedge dg \pm \star(df\wedge dg)=A_{\pm}(dx\wedge dy+\pm du\wedge dv)\\
+ B_{\pm}(dx\wedge du \mp dy\wedge dv) + C_{\pm}(dx\wedge dv \pm
dy\wedge du),
\end{multline}
in terms of linear coordinates $(x,y,u,v)$ on $\R^4$, non-degeneracy
means that $(A_{\pm},B_{\pm},C_{\pm})\from\d
D^4\setminus\{\zero\}\to\R^3$ avoids $\zero$, and so has a
well-defined Hopf invariant. Define
$\lambda(F)=\Hopf(A_{+},B_{+},C_{+})$,
$\rho(F)=\Hopf(A_{-},B_{-},C_{-})$.

\begin{lemma}
$\lambda(F)$ and $\rho(F)$, so defined, equal $\lambda(F)$ and
$\rho(F)$ as defined in \cite{Rudolph:isocp-1}.
\end{lemma}
\begin{proof}
Let $G$ be the Grassmann manifold of oriented $2$-planes in
(oriented) $\R^4$. In \cite{Rudolph:isocp-1}, $(\lambda(F),\rho(F))$
is defined to be the homotopy class of the field of oriented
$2$-planes $\ker(DF)$, identified with a pair of integers via a
standard choice of homeomorphism between $G$ and $S^2\times S^2$.
The Pl\"ucker coordinates of $\ker(DF)$ are the $2\times 2$ minor
determinants of the $2\times 4$ matrix that represents $DF$ in
coordinates $(x,y,u,v)$, and each of the two factors of the map
$\ker(DF)$ from a punctured neighborhood of $\zero$ to $S^2\times
S^2$ is a certain sum or difference of two Pl\"ucker coordinates.
For instance, the factor whose Hopf invariant is $\lambda$ has
coordinates $f_x g_y - f_y g_x + f_u g_v - f_v g_u$, $f_x g_u - f_y
g_v - f_u g_x + f_v g_y$, and $f_x g_v + f_y g_u - f_u g_y - f_v
g_x$, which by inspection are identical, respectively, to $A_{+}$,
$B_{+}$, and $C_{+}$; and similarly for $\rho$.\end{proof}

In particular, if $F$ is tame, then by \cite{Rudolph:isocp-1}
\begin{equation}
\lambda(F)+\rho(F)=\mu(L_{\mathbf{b}_F})
\end{equation}
is the \bydef{Milnor number} of the fibered link $L_{\mathbf{b}_F}$
associated to $F$, that is, the first Betti number of its fiber
surface. Thus $\lambda$, $\rho$, and $\mu$ are, among them, only two
independent invariants; following \cite{Neumann-Rudolph:unfoldings},
we put $\rho$ aside, calling $\lambda(F)$ the \bydef{enhancement of
the fibered link $L_{\mathbf{b}_F}$} (or of its Milnor number), and
writing $\lambda(L)$ for $\lambda(F)$.

The enhancement of a fibered link has been shown to have a number of
interesting and useful properties, summarized in the following
proposition.  For proofs, see \cite{Rudolph:handbook} and references
cited therein.

\begin{proposition}\label{enhancement properties}
\begin{asparaenum}[\upshape (A)]
Let $S$ be a fiber surface in $S^3$, $L$ the fibered link $\Bd S$.
\item\label{not determined by Seifert}
The enhancement of $L$ is not determined by its homological
monodromy, and in particular it is not determined by the Seifert
form of $S$.
\item\label{mirror image}
Let $\mathop{Mir}(L)$ denote the mirror image of $L$. Then
$\lambda(\mathop{Mir}(L))+\lambda(L)=\mu(L)$.
\item\label{0 on positive Hopf}
Let $S$ be a positive Hopf annulus \textup(so that $L$ is two fibers
of a Hopf fibration $S^3\to S^2$, oriented to have linking number
${+}1$\textup).  Then $\lambda(L)=0$. Equivalently, if
$F\from\C^2\to\C\suchthat(z,w)\mapsto zw$, then $\lambda(F)=0$.
\item\label{1 on negative Hopf}
As follows immediately from \eqref{mirror image} and \eqref{0 on
positive Hopf}, if $S$ is a negative Hopf annulus, then
$\lambda(L)=1$. Equivalently, if
$F\from\C^2\to\C\suchthat(z,w)\mapsto z\overline{w}$, then
$\lambda(F)=1$.
\item\label{0 on holomorphic}
More generally than \eqref{0 on positive Hopf}, if
$F\from(\C^2,\zero)\to(\C,0)$ is a complex polynomial without
repeated factors, then $\lambda(F)=0$.
\item\label{Murasugi additive}
If $S$ is a Murasugi sum of fiber surfaces $S_1$ and $S_2$, then
$\lambda(L)=\lambda(\Bd(S_1))+\lambda(\Bd(S_2))$.
\item\label{unfolding additive}
More generally than \eqref{Murasugi additive}, if $F$
\bydef{unfolds} into $F_1$ and $F_2$, in the sense of
\cite{Neumann-Rudolph:unfoldings}, then
$\lambda(F)=\lambda(F_1)+\lambda(F_2)$.
\item\label{Hopf plumbed}
As follows immediately from \eqref{0 on positive Hopf}, \eqref{1 on
negative Hopf}, and \eqref{Murasugi additive}, if $S$ is a
Hopf-plumbed surface, then $\lambda(L)$ is the number of negative
Hopf plumbands; in particular, in this case
$0\le\lambda(L)\le\mu(L)$.
\item\label{Giroux number}
If a contact structure is associated to $L$ by Giroux's construction
\cite{Giroux,GirouxGoodman}, then $\lambda(L)$ is the Hopf invariant
of that contact structure; in particular, $\lambda(L)$ can be any
integer.
\item\label{more Giroux}
Generalizing \eqref{Hopf plumbed}, $\lambda(L)$ is the net number of
negative Hopf plumbands in any stable Hopf plumbing of $S$
\textup(which exists, by \cite{Giroux,GirouxGoodman}\textup).
\item\label{Hirasawa}
Let $e(\beta)$ denote the exponent sum of a braid $\beta\in B_n$.
Given any $\beta$, construct a link $L$ by adjoining to the closed
braid $\widehat\beta$ the braid axis together with, for each
component $K$ of $\widehat\beta$, a nearby oppositely-oriented
longitude $-K'$, where the linking number of $K$ and $K'$ is
unrestricted.  It is a beautiful observation of Hirasawa \textup(see
\cite{HirasawaRudolph}\textup) that then $L$ is fibered and
$\lambda(L)$ is equal to $n-e(\beta)+1$, the Bennequin number of
$\beta$. \qed
\end{asparaenum}
\end{proposition}

\section{Integral formulas for the enhancement}

Applying the results in \textsection\ref{WIF and ICP} to the
definitions in \textsection\ref{lambda}, we can obtain integral
formulas for the enhancement.

If one combines
\begin{equation}\label{omega abc}
\Omega_{(A_{+},B_{+},C_{+})}=\frac{\displaystyle%
                       A_{+}\,dB_{+}\wedge dC_{+}+
                       B_{+}\,dC_{+}\wedge dA_{+}+
                       C_{+}\,dA_{+}\wedge dB_{+}}%
                       {\displaystyle(A_{+}^2+B_{+}^2+C_{+}^2)^{3/2}}
\end{equation}
with the formulas for $A_{+}$, $B_{+}$, and $C_{+}$ given in the
Lemma, the easily confirmed formula
\begin{multline*}
A_{+}^2+B_{+}^2+C_{+}^2=
          (f_x^2+ f_y^2+ f_u^2+ f_v^2)(g_x^2+ g_y^2+ g_u^2+ g_v^2)\\
           - (f_x g_x + f_y g_y + f_u g_u + f_v g_v)^2,
\end{multline*}
and \eqref{WIF}, then---provided one has a method to produce a
$1$-form $\eta_{(A_{+},B_{+},C_{+})}$ (there is no shortage of such
methods, some more explicit than others)---one can derive a
fearsome-looking formula for $\lambda(F)$ as the integral of what
appears to me (perhaps wrongly) to be a pointwise-meaningless
$3$-form, which I forbear to display here.

Pressing on, one might take heart from the thought that the $2$-form
$\Omega_{(A_{+},B_{+},C_{+})}$ seems to be intimately related to the
fibers of $F$ over which
$\eta_{(A_{+},B_{+},C_{+})}\wedge\Omega_{(A_{+},B_{+},C_{+})}$ is
integrated in formula \eqref{ICPIF--first form}.  Although this
formula might be manipulable into something palatable, I have not
yet been able to establish anything along those lines.

A similarly regrettable situation, of high hopes dashed by
intractable computations, prevails in the case of open books and
formula \eqref{OBIF--first form}.

\begin{question}\label{meaningful formulas?}
Are there, in fact, reasonable and meaningful integral formulas for
$\lambda$ in the general case? Such formulas do exist for $\mu$.
For instance, in the complex case, and in general dimensions,
Kennedy \cite{Kennedy} defines a ``twisted tangent bundle'' $E_\d$
over the fiber $F^{-1}(\d)$ ($\d\ne 0$), with $m^\textit{th}$ Chern
form $c_m(E_\d)$, and (simplifying a formula of Griffiths
\cite{Griffiths}) proves that
\begin{equation*}
1-\mu(L) = \lim_{\e\to 0} \lim_{\d\to 0} %
\int\limits_{\displaystyle{F^{-1}(\d)\cap D^{2m+2}_\e}} c_m(E_\d).
\label{Kennedy's integral formula}
\end{equation*}
See also \cite{Langevin,Langevin:1981,Langevin:1982}. In all these
cases, the integrands are ``pointwise-meangingful'' insofar as they
can be interpreted as various kinds of ``curvatures''. Does some
such curvature interpretation exist for general isolated critical
points?
\end{question}

\section{Further questions}
Even if no ``meaningful'' integral formula for $\lambda$ can be
found in the general case, perhaps there are interesting special
cases where one exists. In particular, in light of items \eqref{Hopf
plumbed} and \eqref{Giroux number} in Proposition~\ref{enhancement
properties}, it seems natural to ask for a (geometric or other)
characterization of those fibered links $L$ for which
$0\le\lambda(L)\le\mu(L)$.

\begin{questions}\label{special cases}
\begin{inparaenum}[(a)]
\item
Can this class of fibered links be characterized using integral
formulas?
\item\label{minimal}
Is this class of fibered links perhaps exactly those for which a
mapping $F$ can be found such that all fibers of $F$ are minimal
surfaces in $\R^4$?
\end{inparaenum}
\end{questions}

\begin{remark}
Question~\eqref{minimal} is emphatically not a
conjecture.  The main evidence for the affirmative answer is that,
if $F$ is a complex polynomial (or the mirror image of a complex
polynomial), then all its fibers, being complex curves, are minimal
surfaces.  
Some reason to think that the
techniques of this paper might at least be relevant is that,
according to Hoffman \& Osserman \cite{HoffmanOsserman}, a necessary
and sufficient condition for all the fibers of $F\from
(U,\zero)\to(\R^2,\zero)$ to be minimal is that $\d
D^4\setminus\{\zero\}\to G\suchthat \mathbf{x}\mapsto \ker(df\wedge
dg)$ be fiberwise (anti-)conformal, which puts strong restrictions
on $(A_{+},B_{+},C_{+})$ and its anti-self-dual analogue
$(A_{-},B_{-},C_{-})$.
There is a considerable literature on topological restrictions of
foliations by minimal surfaces, 
some of which might perhaps be adapted to the investigation of 
this question; in particular, ``Rummler's Calculation'' 
(\cite{Rummler}, cited by Sullivan \cite{Sullivan}) seems 
highly apposite. 
\end{remark}

In 1988, before Walter Neumann used the calculus of splice diagrams
to work out an example (published in
\cite{Neumann-Rudolph:unfoldings}) of a fibered link $L$ with
$\lambda(L)<0$---and well before I noticed that, in fact, such links
are plentiful among a family of examples already introduced in
\cite{Rudolph:isocp-1}---, it seemed reasonable (particularly in
light of item \eqref{Hopf plumbed} of Proposition~\ref{enhancement
properties}) to wonder whether $\lambda(L)$ might be the dimension
of some vectorspace naturally associated to $L$.  At that time,
Isadore Singer suggested to me that, even if it were never negative,
it might nonetheless be better interpreted as the Euler
characteristic of some complex (or perhaps the index of some
operator), rather than as a dimension.  Both suggestions seem even
better to me now than they did then, but in 20 years I have made no
progress on them.  I do have, based on little more than wishful
thinking, the notion that such a complex might be constructed by
taking note of the complex structures naturally imposed on the
fibers of $F$ by their induced metrics; note the connection to the
Hoffman--Osserman generalized Gauss map.

\begin{question}
Can some reader do better with Singer's suggestions?
\end{question}

\providecommand{\bysame}{\leavevmode\hbox
to3em{\hrulefill}\thinspace}
\providecommand{\MR}{\relax\ifhmode\unskip\space\fi MR }
\providecommand{\MRhref}[2]{%
  \href{http://www.ams.org/mathscinet-getitem?mr=#1}{#2}
} \providecommand{\href}[2]{#2}

\bigskip
\noindent Lee Rudolph\\
Department of Mathematics\\ Clark University, Worcester MA 01610
USA\\
E-mail: lrudolph@black.clarku.edu
}
\end{document}